\documentclass[11pt,reqno]{amsart}
\usepackage{setspace}
\usepackage{footnote}
\usepackage{graphicx}
\usepackage{cmap,mathtools}
\usepackage[T1]{fontenc}
\usepackage[utf8]{inputenc}
\usepackage[english]{babel}
\usepackage{amsfonts,amssymb,amsthm,amsmath,enumerate,bm,cite,lmodern}
\usepackage{url}
\usepackage[shortlabels]{enumitem}
\usepackage{float}
\usepackage{secdot}
\mathtoolsset{showonlyrefs} 

\allowdisplaybreaks
\usepackage{epstopdf}
\usepackage{a4wide}
\usepackage{xparse}
\usepackage[dvipsnames]{xcolor}
\usepackage[colorlinks=true,linktocpage,pdfpagelabels,bookmarksnumbered,bookmarksopen]{hyperref}

\hypersetup{
	colorlinks=true,
	linkcolor=blue,         
	citecolor=blue,
	urlcolor=black
}

\usepackage[hyperpageref]{backref}

\usepackage[margin=1.1in]{geometry}

\usepackage{orcidlink}

\usepackage{esint}

\newtheorem{theorem}{Theorem}
\newtheorem{proposition}[theorem]{Proposition}
\newtheorem{lemma}[theorem]{Lemma}
\newtheorem{corollary}[theorem]{Corollary}
\theoremstyle{definition}
\newtheorem{definition}[theorem]{Definition}
\newtheorem{remark}[theorem]{Remark}

\let\OLDthebibliography\thebibliography
\renewcommand\thebibliography[1]{
	\OLDthebibliography{#1}
	\setlength{\parskip}{1pt}
	\setlength{\itemsep}{1pt plus 0.3ex}
}

\numberwithin{equation}{section}
\numberwithin{theorem}{section}

\DeclarePairedDelimiter\norm{\lVert}{\rVert}%
\makeatletter
\let\oldnorm\norm
\def\norm{\@ifstar{\oldnorm}{\oldnorm*}}

\newcommand{\al} {\alpha}

\newcommand{\pa} {\partial}
\newcommand{\be} {\beta}
\newcommand{\de} {\delta}
\newcommand{\De} {\Delta}
\newcommand{\Dn} {(-\Delta)^{\frac{n}{2}}}

\newcommand{\ga} {\gamma}

\newcommand{\om} {\omega}

\newcommand{\La} {\Lambda}

\newcommand{\Gr} {\nabla}

\newcommand{\noi} {\noindent}
\newcommand{\ep} {\epsilon}
\newcommand{\ra} {\rightarrow}

\newcommand{\vp}{\varphi}

\newcommand\restr[2]{{
  \left.\kern-\nulldelimiterspace 
  #1 
  \right|_{#2} 
  }}

\def\C{{\mathcal C}}

\def\Q{{\mathcal Q}}
\def\X{{\mathcal X}}

\def\S{\mathbb{S}}

\def\R{{\mathbb R}}

\def\d{{\rm d}}
\def\dr{{\rm d}r}
\def\dS{{\rm dS}}

\def\dz{{\rm d}z}
\def\dx{{\rm d}x}
\def\dy{{\rm d}y}

\usepackage[foot]{amsaddr}
\usepackage[pagewise]{lineno}

\makeatother

\date{}
\begin{document}
\title[Bol's type inequality for singular metrics]{Bol's type inequality for singular metrics and its application to prescribing $\Q$-curvature problems}


\author[M. Ghosh AND A. Hyder]{Mrityunjoy Ghosh$^1$ AND Ali Hyder$^1$}

\address{$^1$Tata Institute of Fundamental Research,
	Centre for Applicable
	Mathematics\\
	Post Bag No. 6503, Sharadanagar,
	Bangalore 560065, India}

\email{ghoshmrityunjoy22@gmail.com, hyder@tifrbng.res.in}


\begin{abstract}
In this article, we study higher-order Bol's inequality for  radial normal solutions to a singular Liouville equation. 
By applying these  inequalities along with compactness arguments, we derive  necessary and sufficient conditions for the existence of radial normal solutions to  a singular $\Q$-curvature problem. Moreover, under suitable assumptions on the $\Q$-curvature, we obtain  uniform bounds on the total $\Q$-curvature. 

\end{abstract}

\maketitle


\section{Introduction}\label{Section}

We consider the following singular $\Q$-curvature equation \begin{align}\label{eq-10}\left\{\begin{array}{ll}\Dn \tilde u=Qe^{n\tilde u}-\tilde\alpha\delta_0,\\ \rule{.0cm}{.5cm}\int_{\R^n}|Q|e^{n\tilde u}\dx<\infty,\end{array}\right. \end{align} where $n\geq 2$, $Q\in L^\infty_{\rm loc}(\R^n)$  is prescribed, $\tilde\alpha\in\R$ and $\delta_0$ is the Dirac mass at the origin. As $$(-\Delta)^\frac n2 \log\frac{1}{|x|}=\gamma_n\delta_0,\quad \gamma_n:=\frac{(n-1)!}{2}|\S^n|,$$ setting $$u(x):=\tilde u(x)+\frac{\tilde\alpha}{\gamma_n}\log\frac{1}{|x|},$$ we see that $u$ satisfies the following $\Q$-curvature equation in $\R^n$: 
\begin{equation}\label{Q-equation}
	\left\{\begin{aligned}
		\Dn u &= |x|^{n\al}Q e^{nu}\quad \text{in}\; \R^n,\\
		\int_{\R^n}|x|^{n\al} & |Q(x)| 
 e^{nu(x)}\dx< \infty,
	\end{aligned}\right.
\end{equation}
where  $\alpha:=\frac{\tilde\alpha}{\gamma_n}$. When $n\in\mathbb N$ is an odd integer, the operator $(-\Delta)^\frac n2$ is nonlocal, and therefore, in order to define the above equation in the distributional sense, one has to assume certain growth assumptions on $u$ at infinity. As we will be working mainly with the   integral equation \eqref{eq-normal} below, we omit the definition and basic properties of the nonlocal operator $(-\Delta)^\frac n2$ here, and refer the interested reader to \cite{DaLioMartinazzi2015,Ali2019DIE,MartinazziJin,KonigLaurain} and the references therein. 

 Throughout this article, we assume that $\alpha>-1$ (unless specified elsewhere), and the volume of the conformal metric $e^{2\tilde u}|\dx|^2$ is finite, i.e., 
\begin{equation}\label{Volume}
	\La:=\int_{\R^n}e^{n\tilde u}\dx= \int_{\R^n} |x|^{n\al} e^{nu(x)}\dx<\infty,
\end{equation} where $|\dx|^2$ denotes the Euclidean metric. 
Geometrically, if $u$ is a solution of  \eqref{Q-equation}, then the $\Q$-curvature (cf. \cite{Branson_Orsted,Chang,Fefferman2002,Fefferman2003, GJMS})  
of the conformal metric $e^{nu}|\dx|^2$  is  $|x|^{n\al}Q.$ In particular, when $n=2,$ the Gaussian curvature of the metric $e^{nu}|\dx|^2$ is $|x|^{n\al}Q.$ 

Existence and classification of solutions to \eqref{Q-equation}   is extensively studied by many authors in the past decades with various assumptions on the $\Q$-curvature; to quote a few, we mention the articles by, Da Lio-Martinazzi-Rivi\`{e}re \cite{DaLioMartinazzi2015}, DelaTorre-Mancini-Pistoia \cite{DelaTorre2020} and Ahrend-Lenzmann \cite{AhrendLenzmann} for $n=1$;  
 Chen-Li \cite{Chen1991}, Troyanov \cite{Troyanov}, Chanillo-Kiessling \cite{Chanillo1994}, Li-Shafrir \cite{LiShafrir1994}, Prajapat-Tarantello \cite{PrajapatTarantello} and Del Pino-Esposito-Musso \cite{DEM2012} for $n=2$; Jin-Maalaoui-Martinazzi-Xiong \cite{MartinazziJin} for $n=3$; Lin \cite{Lin1998}, Wei-Ye \cite{WeiYe2008}, Hyder-Martinazzi \cite{Ali_Martinazzi},  Jin-Shu-Tai-Wu \cite{Jin_Shu} and Ahmedou-Wu-Zhang \cite{Ahmedou} for $n=4$; Wei-Xu \cite{WeiXu1999}, Chang-Chen \cite{ChangChen2001}, Xu \cite{Xu2005}, Martinazzi \cite{Martinazzi2013AIH,Martinazzi2009}, Huang-Ye \cite{HuangYe}, Hyder \cite{Ali2019DIE,Ali2017APDE},  Hyder-Mancini-Martinazzi  \cite{Ali_Mancini}, K\"onig-Laurain \cite{KonigLaurain} 
 for higher dimensions. For further developments in this direction, we refer the reader to \cite{HuangYe2025, Malchiodi2008,Malchiodi2007,JostWang2009,Gursky2015} and the references  therein.

Before proceeding further, we recall the definition of normal solutions of \eqref{Q-equation}. 
\begin{definition}\label{Def:normal}
    A solution $u$ of \eqref{Q-equation} is called a \textit{normal solution} if  $u$ satisfies the following integral equation 
    \begin{align}\label{eq-normal}u(x)=\frac{1}{\ga_n}\int_{\R^n} \log\left(\frac{1+|y|}{|x-y|}\right)|y|^{n\al} Q(y)e^{nu(y)}\dy +c,\end{align} for some $c\in\R$. 

\end{definition}

It is well-known that for $Q\equiv1$, the volume $\Lambda$ associated to any normal    solution $u$ to \eqref{Q-equation} is quantized. More precisely,   $$\La=\La_1(1+\al),\quad \La_1:=(n-1)! |\mathbb{S}^n|.$$ Interestingly, for the regular case $\alpha=0$ in two dimensions, the above equality can be seen as an optimal case of the Alexandrov-Bol's inequality. Let us recall first the following version of Alexandrov-Bol's inequality from \cite{Bandle1976,Suzuki}  (for various other forms of Bol's type inequalities and its applications we refer the reader to \cite{Bol1941,Bartolucci2019, Gui2018,Gui2023} and the references therein): if   $\psi\in C^2(\Omega)\cap C^0(\bar\Omega)$ is positive for some smooth open set $\Omega\subset\R^2$ and $\psi$ satisfies \begin{align}\label{eq-psi} -\Delta\log\psi\leq \psi\quad\text{in }\Omega ,\end{align} then \begin{align}\label{eq-AB}  \left(\int_{\partial\Omega} \sqrt\psi \d\sigma\right)^2\geq\frac12\left(8\pi-\int_{\Omega} \psi \dx\right)\int_{\Omega}\psi \dx.\end{align}  For the choice $\psi=2e^{2u}$,  one has \eqref{eq-psi} in $\R^2$ provided $Q\leq1$, and in particular, taking $\Omega=B_R$ with $R\to\infty$ in \eqref{eq-AB}, one deduces that $\Lambda\geq4\pi=\Lambda_1$.

In \cite{WeiLi2023}, Li-Wei obtained the following  Bol's-type inequality in higher dimensions: 
\begin{itemize} \item[(i)]  Let $u$ be a radial normal solution of \eqref{Q-equation} in $\R^n$ with  $n\geq 2$, $\al=0$ and   $Q(x)\leq 1$ for all $x\in\R^n$. Then  $	\La\geq \La_1$, and the  equality holds   if and only if $Q\equiv 1.$
\item[(ii)]  Let $u$ be a radial normal solution of \eqref{Q-equation} in $\R^n,$ with $n\geq 2$  and $\al=0.$ Assume that $Q(x)\geq 1$ for all $x\in\R^n$. Then $\La\leq \La_1,$ and the equality holds  if and only if $Q\equiv 1.$
\end{itemize}

In this article, first, we derive an analogous Bol's-type inequality for the singular case, that is, for every $\alpha>-1$, we prove:

\begin{theorem}\label{Bols_up}
    Let $u$ be a radial normal solution to \eqref{Q-equation} for some $\alpha>-1$ and $Q\in L^\infty_{\rm loc}(\R^n)$. 
    Then if $Q\leq 1,$ we have 
    $$\La\geq \La_1(1+\al).$$
Moreover, the equality holds if and only if $Q\equiv 1.$
\end{theorem}

\begin{theorem}\label{Bols_low}
  Let $u$ be a radial normal solution to \eqref{Q-equation} for some $\alpha>-1$ and $Q\in L^\infty_{\rm loc}(\R^n)$. Then if $Q\geq 1,$ we have 
    $$\La\leq \La_1(1+\al).$$
Moreover, the equality holds if and only if $Q\equiv 1.$
\end{theorem}

The proofs of Theorem \ref{Bols_up} and Theorem \ref{Bols_low} rely on appropriate  Pohozaev-type identities involving $\La$. Following the approach in \cite{WeiLi2023}, we first derive a Pohozaev-type inequality (Proposition \ref{Pohozaev}), which is sufficient to prove Theorem \ref{Bols_up}. To obtain  an analogous identity for Theorem \ref{Bols_low}, we begin by observing certain decay properties (Lemma \ref{u_weakBound}) of solutions when $|x|$ is sufficiently large. These properties allow us to control the total curvature outside a sufficiently large ball and, when  combined with an application of Jensen’s inequality (in the same spirit of \cite[Lemma 3.5]{Ali_Mancini}),  yield the estimate stated in Proposition \ref{Pohozaev_equal}. As a result, we establish a Pohozaev-type identity in Proposition \ref{Pohozaev_equal}, which plays a crucial role in the proof of Theorem \ref{Bols_low}.

We now  consider the following singular $\Q$-curvature equation in $\R^n$ ($n\geq 2$): 
\begin{equation}\label{Q1_equation}
	\left\{\begin{aligned}
		\Dn u &= (1+|x|^{n\al}) e^{nu}\quad \text{in}\; \R^n,\;-1<\al<0,\\
		\La_*:= &\int_{\R^n} (1+|x|^{n\al}) 
		e^{nu(x)}\dx< \infty.
	\end{aligned}\right.
\end{equation}

The above equation in dimension four has been studied by Hyder-Martinazzi in \cite{Ali_Martinazzi} (see also \cite{Struwe2020, Struwe2021} for some non-existence results) for $\alpha>0$. They provided a complete description (cf. \cite[Theorem 1.5]{Ali_Martinazzi}) of the range of $\La_*$ for the existence of a radial normal solutions of \eqref{Q1_equation} when $\al\in (0,1]$. Furthermore, for $\al>1,$ the authors gave a partial answer about the possible range of $\La_*,$ and a complete answer to this case is recently obtained by Li-Wei in \cite[Corollary 7.2]{WeiLi2023} with the help of Bol's inequalities.

Applying Theorem \ref{Bols_low}, we establish a necessary and sufficient condition on the total curvature $\La_*$ for which   radial normal solutions to \eqref{Q1_equation} exist.

\begin{theorem}\label{existence_coro}
	Let $-1<\al<0.$ Then there exists a radial normal solution to  \eqref{Q1_equation} if and only if 
	 $$\La_1\max\{-\al, 1+\al\}<\La_*<\La_1.$$
\end{theorem}

Notice that the PDE \eqref{Q1_equation} is supercritical for $\Lambda_*>\Lambda_1(1+\alpha)$ in the sense that the behavior of the nonlinearity $(1+|x|^{n\alpha})e^{nu}$ near the origin is  similar to $|x|^{n\alpha}e^{nu}$ (for $\alpha\in (-1,0)$), and a possible  blow-up phenomena  (in the class of normal solutions)  can not  be ruled-out if $\Lambda_*\geq \Lambda_1(1+\alpha)$. Nevertheless, for each $\rho\in\R$, one can show the existence of  a radial normal solution $u_\rho$ to \eqref{Q1_equation} with $u_\rho(0)=\rho$ by fixed point arguments (see, e.g., \cite{Ali_Martinazzi} for the case $n=4$ and $\alpha>0$).  In dimension $n=4$, as radial normal solutions are uniquely determined by their value at the origin,  it turns out that  the function $\rho\to\Lambda(\rho)$ is continuous. In particular, studying the behavior of $\Lambda(\rho)$ as $\rho\to\pm\infty$, one gets existence of radial normal solutions   for each admissible value of $\Lambda_*$, as  determined by the Bol's inequality.

 The main difficulty in proving the existence results  in our case lies in the fact  that the uniqueness  of radial normal solutions to \eqref{Q1_equation} is not known (for $n=2,4$, uniqueness can be obtained in the spirit of \cite{Ahmedou}). To overcome this difficulty, we establish existence for each prescribed $\Lambda_*$ satisfying $\La_1\max\{-\al, 1+\al\}<\La_*<\La_1$ by means of  fixed point arguments. To this end, we construct \emph{approximate} solutions in such a way that compactness can be achieved even in the supercritical regime, see Proposition \ref{propo-2}.

	\medskip 

 We recall from \cite{Martinazzi2013AIH,HuangYe,Ali2017APDE} that in dimension $n\geq5$, there are  radial conformal metrics with arbitrary total $\Q$-curvature, that is, for every $\Lambda_*\in (0,\infty)$,  \begin{align}\label{def-totalQ} \La_*:=\int_{\R^n}Q|x|^{n\al}e^{nu}\dx,  \end{align}  Eq.  \eqref{Q-equation}  with $Q\equiv 1$, $\alpha=0$, $n\geq5$ admits a radially symmetric solution (which is not normal).  Such solutions can be written in the form  $u=v+p$ for some nontrivial polynomial $p$,   and  $v$ satisfies  the integral equation \eqref{eq-normal} with $\alpha=0$,  $Q=e^{np}$. Moreover, the  conformal metrics $e^{2u}|\dx|^2$ and $e^{2v}|\dx|^2$ have the same total curvature $\Lambda_*$ as given in \eqref{def-totalQ}. This demonstrates that without suitable  control on the $\Q$-curvature, the total curvature associated  with  normal conformal metrics need not admit  either lower or upper bounds.  

Under suitable hypotheses on $Q,$ we prove  a  uniform bound on the total $\Q$-curvature $\Lambda_*$, which rules out the unbounded behavior described above. 

\begin{theorem}\label{total-curvature}
	 Let $Q\in L^\infty_{\rm loc}(\R^n)$ be a radial function such that 
	$$Q\geq 1+|x|^p\;\quad and\; 
		(Q-M|x|^p)^+\in L^1(\R^n),$$
	  for some $p\geq0$, $M\geq1$. 
	Then for any $\alpha>-1$ there exists $$C:=C\left(n,p,M,\alpha,\|Q\|_{L^\infty(B_1)},\|(Q-M|x|^p)^+\|_{L^1(B_1^c)} \right)>0$$  
	such that  for any radial normal solution $u$ of \eqref{Q-equation}, we have the total curvature 
	$$\La_*\leq C.$$
\end{theorem}

Let us emphasize here that  a lower bound on $Q$ alone, namely $Q\geq 1+|x|^p$ is not sufficient to conclude the above theorem, see Remark \ref{example:Q}  below. It would be interesting to determine  whether  the above theorem remains valid when the constant $C$ depending on the  $L^1$ norm of $Q$ rather than its   $L^\infty$ norm.

As an another  application of Theorem \ref{Bols_low}, we prove the following existence result:  

\begin{theorem} \label{f_existence}
	Let $f\in   L^1(\R^n)$ be a non-negative continuous radially symmetric function on $	\R^n.$  Then for every $ \rho\in\R,$ there exists a radial normal solution $u_\rho$ to \eqref{Q-equation}  with $Q=1+f$ and $u_\rho(0)=\rho$.  Moreover, $$ \int_{\R^n}Qe^{nu_\rho}\dx\to \La_1(1+\al)\text{ as }\rho\to\pm\infty.$$
	\end{theorem}
	
	For each $\rho\in\R$, the existence of an  \emph{approximate} radial normal solution $u_{\rho,\ep}$ to \eqref{Q-equation} with $Q=1+f$, $u_{\rho,\ep}(0)=\rho$ follows in a standard way (see, e.g., \cite{Ali_Martinazzi}). To obtain compactness of the solutions $u_{\rho,\ep}$ as $\ep\to0$, one requires that the total $\Q$-curvature corresponding to these normal solutions be uniformly bounded. In general, such a   uniform bound on the total $\Q$-curvature is obtained using  Pohozaev-type identity (see, e.g.,  \cite[Lemma 6.5]{Ali_Martinazzi}, \cite[Eq. (20)]{WeiYe2008}), where the sign, or the precise structure  of $x\cdot\nabla Q$ plays an important role. However, in our setting,  under the hypothesis that  $f\in L^1(\R^n)$, we do not have  sufficient  control over the term $x\cdot \nabla f$ (even if $f$ is $C^1$). Nevertheless, the uniform volume bound established in Theorem \ref{Bols_low} compensates for the lack of control on $x\cdot\nabla f$, and yields the desired compactness.

{}
The rest of the article is organized as follows. 
 Section \ref{Sec:main_result} consists of various estimates related to the asymptotic behaviour of the solutions, and we derive some Pohozaev-type identities/inequalities  in this section. The proofs of the Bol's inequalities (Theorems \ref{Bols_up} and \ref{Bols_low}) are the contents of Section \ref{Sec:proof_Bol}. We prove Theorem \ref{existence_coro} in Section \ref{Sec:existence} and Theorem \ref{f_existence} in Section \ref{Sec:existence_f}. Section \ref{Sec:bound_proof}  is devoted to the proof of Theorem \ref{total-curvature}.

\section{Decay estimates and Pohozaev identities}\label{Sec:main_result}
In this section, we establish various decay estimates of the solutions of \eqref{Q-equation}. Throughout the article,  $C$ denotes  a generic positive constant whose value  may change from line to line. We write  $B_r(x)$ for  the ball of radius $r>0$ centered at $x$;  when the center is the origin, we simply write $B_r$ instead of $B_r(0).$ 

Since $Q\in L^\infty_{\rm loc}(\R^n)$, for each $\alpha>-1$, any normal solution of \eqref{Q-equation} will be in $C^s_{\rm loc}(\R^n)\cap C^1(\R^n\setminus\{0\})$. This fact will be used throughout the article. 

\subsection{Decay properties} 

For a normal solution   $u$   of  \eqref{Q-equation} we set  
\begin{equation}\label{def:beta}
	\be := \frac{1}{\ga_n}\int_{\R^n} |x|^{n\al} Q (x)e^{nu(x)}\dx.
\end{equation}

\begin{lemma}\label{u_UpLow}
	There exists $R>>1$ such that \begin{enumerate}[(i)]
		\item if $Q^+$ has compact support,  then 
		$$u(x)\leq -\be\log(|x|)+C,\;\text{for all}\;|x|> R,$$
		
		\item if $Q^-$ has compact support, then 
		$$u(x)\geq -\be\log(|x|)-C,\;\text{for all}\;|x|> R,$$
	\end{enumerate}
	where $Q^{\pm}=\max\{\pm Q, 0\}$ and $C$ is some constant.
\end{lemma}
\begin{proof} The proof is quite standard  and we leave it to the reader. 
	
\end{proof}

\begin{corollary}\label{lamda_al}
	Let $\be$ be defined by \eqref{def:beta}. If $Q\geq 1$, then we have  $\be>1+\al.$
\end{corollary}
\begin{proof}
	By assumption, $Q^-\equiv 0.$ Therefore, using Lemma \ref{u_UpLow}-$(ii),$ we get for $|x|\geq R>>1$  
	\begin{align}
\int_{B_{\frac{|x|}{2}}(x)}|y|^{n\al} e^{nu(y)}\dy &\geq C \int_{B_{\frac{|x|}{2}}(x)}\frac{1}{|y|^{n\left(\be-\al\right)}}\dy 
		\geq  \frac{C}{|x|^{n\left(\be-\al-1\right)}}. 
	\end{align}
	Since $Q\geq 1$ and $Q|x|^{n\alpha}e^{nu}\in L^1(\R^n)$, we have $$\int_{B_{\frac{|x|}{2}}(x)}|y|^{n\al} e^{nu(y)}\dy\to0\quad\text{ as }|x|\to\infty ,$$ and hence,	$\be>1+\al.$
\end{proof}

A proof of the following lemma can be found in \cite[Lemma 2.4]{Lin1998}. 
\begin{lemma}\label{u_weakBound}
	Let $Q\geq 0$.  Then for $\ep>0$ there exists $R>>1$ such that 
	\begin{equation}
		u(x) \leq \left(-\be+\ep\right)\log(|x|)+ \frac{1}{\ga_n}\int_{B_1(x)}\log\left(\frac{1}{|x-y|}\right)|y|^{n\al} Q(y)e^{nu(y)}\dy,
	\end{equation}
	for all $|x|\geq R.$
\end{lemma}

\subsection{Pohozaev-type identities}
In this subsection, we derive some Pohozaev-type  identities (or inequalities)  associated with  normal solutions  of  \eqref{Q-equation}.
We define 
\begin{equation}\label{v_standard}
	v(x)=\frac{1}{\ga_n}\int_{\R^n}\log \left(\frac{1+|y|}{|x-y|}\right)|y|^{n\al} e^{nu(y)}\dy.
\end{equation}

\begin{lemma}\label{lem:bound_gradv}
	Let $u$ be a normal solution of \eqref{Q-equation} for some $Q\in L^\infty_{l\rm oc}(\R^n)$, and let $v$ be as defined in \eqref{v_standard}. Then there exists $\delta >0$ such that for  every $0<|x|\leq 1,$ we have 
	\begin{equation}
	|\nabla u(x)|+	|\nabla v(x)|\leq C\left(1+|x|^{\de-1}\right).
	\end{equation}
\end{lemma}

\begin{proof}
	Let $x\in \R^n$ such that $0<|x|\leq 1.$ We decompose $\R^n$ as $\R^n=D_1\cup D_2\cup D_3,$ where 
	\begin{align}
		D_1:=\left\{y\in\R^n: |y-x|<\frac{|x|}{2}\right\}, \;D_2:=\left\{y\in \R^n: |y|\geq 2|x|\right\},\;D_3:=\R^n\setminus (D_1\cup D_2).
	\end{align}
	From the definition \eqref{v_standard} of $v$, one has
	\begin{align}
		|\nabla v(x)| &\leq C \int_{\R^n}\frac{1}{|x-y|} |y|^{n\al}e^{nu(y)}\dy.
	\end{align}
	Now, since $u\in C^0(\R^n)$,
	we obtain 
	\begin{align}
		I_1:=\int_{D_1}\frac{1}{|x-y|} |y|^{n\al}e^{nu(y)}\dy & \leq C|x|^{n\al} \int_{D_1}\frac{1}{|x-y|}\dy 
		\leq   
		C  |x|^{n\al+n-1}.
	\end{align}
	Similarly, 
	\begin{align}
		I_2:=\int_{D_2}\frac{1}{|x-y|} |y|^{n\al}e^{nu(y)}\dy & =\left\{\int_{\{ |y|\leq 1 \}\cap D_2}+\int_{\{|y|> 1\}\cap D_2}\right\} \frac{1}{|x-y|} |y|^{n\al}e^{nu(y)}\dy\\
		& \leq C \int_{2|x|\leq |y|\leq 1} |y|^{n\al-1}\dy+C\\  &\leq C(1+|x|^{\delta-1}), 
	\end{align} provided $\delta\in (0, n\alpha+n)$. 
	Proceeding in the same way, one can derive that 
	\begin{equation}
		I_3:=\int_{D_3}\frac{1}{|x-y|} |y|^{n\al}e^{nu(y)}\dy\leq  C(1+|x|^{\delta-1}). 
	\end{equation}
Combining the above estimates we conclude   $|\nabla v(x)|\leq C(1+|x|^{\delta-1}).$  
Since $Q\in L^\infty_{\rm loc}(\R^n)$,  analogous estimates for $u$ can be obtained in a similar way. 
\end{proof}

Proof of the following proposition is in the spirit of \cite{WeiLi2023,  Xu2005}.

\begin{proposition}\label{Pohozaev}
    Let $u$ be a normal solution of \eqref{Q-equation} with $Q\in L^\infty_{\rm loc}(\R^n)$. Suppose $v$ is as defined by \eqref{v_standard}. Then there exists a sequence $\{R_k\}$ with $R_k\ra \infty$ such that the following identity holds:
    \begin{equation}\label{eq:pohozaev}
    \frac{\La}{2\ga_n}(\La-2\ga_n)\geq \frac{1}{n}\limsup_{k\ra\infty}\int_{B_{R_k}}(x\cdot \nabla Q_0(x)) e^{nv(x)}\dx,
    \end{equation}
    where $\La$ is given by \eqref{Volume} and $Q_0(x):=|x|^{n\al}e^{n(u-v)(x)}.$
\end{proposition}
\begin{proof}
If $\La=+\infty,$ the statement follows. Assume that $\La<+ \infty.$ Since $v$ satisfies \eqref{v_standard}, we have  
    \begin{equation}\label{grad_v}
	\Gr v(x)=-\frac{1}{\ga_n}\int_{\R^n} \frac{(x-y)}{|x-y|^2} Q_0(y) e^{nv(y)}\dy, 
	\end{equation}
    where $Q_0$ is as given in the statement. Taking a dot product in \eqref{grad_v} with $ \left(Q_0(x)e^{nv(x)}\right)x$ and integrating the resultant on $B_R\setminus B_\ep$ ($\ep>0$), we get
    \begin{align}
    	I_{1,\ep}:= \int_{B_R\setminus B_\ep} (x \cdot\Gr v(x)) & Q_0(x) e^{nv(x)} \dx \\
    	& = -\frac{1}{\ga_n}\int_{B_R\setminus B_\ep}\int_{\R^n} \frac{x\cdot (x-y)}{|x-y|^2} Q_0(y) Q_0(x) e^{nv(y)} e^{nv(x)} \dy\dx\nonumber\\
    	&:=I_{2,\ep}.\label{eq:I_epsilon}
    \end{align}
    Using divergence theorem, we derive  
    \begin{equation}
    	I_{1,\ep}=-\int_{B_R\setminus B_\ep} Q_0(x) e^{nv(x)}\dx -\frac{1}{n}\int_{B_R\setminus B_\ep}(x\cdot \nabla Q_0(x))e^{nv(x)}\dx +\frac{R}{n}\int_{\pa (B_R \setminus B_\ep)} Q_0(x) e^{nv(x)}\dS(x).
    \end{equation}
    Applying Lemma \ref{lem:bound_gradv}, we deduce $I_{1,\ep}\rightarrow I_1$ as $\ep\rightarrow 0,$ where
    \begin{equation}\label{Div_I1}
    	I_1=-\int_{B_R} Q_0(x) e^{nv(x)}\dx -\frac{1}{n}\int_{B_R}(x\cdot \nabla Q_0(x))e^{nv(x)}\dx +\frac{R}{n}\int_{\pa B_R} Q_0(x) e^{nv(x)}\dS(x).
    \end{equation}
    Similarly, we can show that $I_{2,\ep}\rightarrow I_2$ as $\ep\rightarrow 0,$  where 
    \begin{equation}
    	I_2:=-\frac{1}{\ga_n}\int_{B_R}\int_{\R^n} \frac{x\cdot (x-y)}{|x-y|^2} Q_0(y) Q_0(x) e^{nv(y)} e^{nv(x)} \dy\dx.
    \end{equation}
    Thus, \eqref{eq:I_epsilon} implies
    \begin{align}
		I_1 & =I_2.\label{Inte}
	\end{align}
 Now,
 \begin{align}\label{Div_I2}
     I_2 &=-
     \frac{1}{2\ga_n}\int_{B_R}\int_{\R^n}Q_0(y)Q_0(x)e^{nv(y)}e^{nv(x)}\dy\dx\nonumber\\
     & \hspace{3cm} -\frac{1}{2\ga_n}\int_{B_R}\int_{\R^n}\frac{(x+y)\cdot (x-y)}{|x-y|^2}Q_0(y)Q_0(x)e^{nv(y)}e^{nv(x)}\dy\dx\nonumber\\
     &=-
     \frac{1}{2\ga_n}\int_{B_R}\int_{\R^n}Q_0(y)Q_0(x)e^{nv(y)}e^{nv(x)}\dy\dx\nonumber\\
     & \hspace{3cm} -\frac{1}{2\ga_n}\int_{B_R}\int_{B_R^c}\underbrace{\frac{(x+y)\cdot (x-y)}{|x-y|^2}Q_0(y)Q_0(x)e^{nv(y)}e^{nv(x)}}_{:=F(x,y)}\dy\dx,
 \end{align}
 where we use the anti-symmetry of the integrand to get the last equality. Now, we show that
 \begin{align}\label{Double_zero}
     \limsup_{R\ra \infty}\int_{B_R}\int_{B_R^c} F(x,y)\dy\dx & 
     \leq  0.
 \end{align}
To prove the assertion, we rewrite the double integral on $B_R\times B_R^c$ in the following way:
\begin{align}
\int_{B_R}\int_{B_R^c}F(x,y)\dy\dx & =\underbrace{\int_{B_{\frac{R}{2}}} \int_{B_R^c}F(x,y)\dy\dx}_{:=J_1}\nonumber\\
&+\underbrace{\int_{B_R\setminus B_{\frac{R}{2}}}\int_{B_{2R}^c} F(x,y)\dy\dx}_{:=J_2}+\underbrace{\int_{B_R\setminus B_{\frac{R}{2}}}\int_{B_{2R}\setminus B_R} F(x,y)\dy\dx}_{:=J_3}.\label{Double_decom}
\end{align}
Since $\La=\int_{\R^n} Q_0(z) e^{nv(z)} \dz<+\infty$ (by assumption), we observe that
\begin{align}
    |J_1|& \leq C \int_{B_{\frac{R}{2}}} \int_{B_R^c} Q_0(y)Q_0(x)e^{nv(y)}e^{nv(x)} \dy\dx\ra 0\;\text{as}\;R\ra\infty\label{J1_zero}\\
    \text{and}\;|J_2| & \leq C \int_{B_R\setminus B_{\frac{R}{2}}}\int_{B_{2R}^c} Q_0(y)Q_0(x)e^{nv(y)}e^{nv(x)} \dy\dx\ra 0\;\text{as}\;R\ra\infty.\label{J2_zero}
\end{align}
Furthermore, since $Q_0(z)e^{nv(z)}=|z|^{n\al}e^{nu(z)}$ has constant sign, we deduce
\begin{align}
    J_3 &= \int_{B_R\setminus B_{\frac{R}{2}}}\int_{B_{2R}\setminus B_R} \frac{(x+y)\cdot (x-y)}{|x-y|^2}Q_0(y)Q_0(x)e^{nv(y)}e^{nv(x)}\dy\dx \nonumber\\
    &=\int_{B_R\setminus B_{\frac{R}{2}}}\int_{B_{2R}\setminus B_R} \frac{|x|^2-|y|^2}{|x-y|^2}|y|^{n\al}|x|^{n\al}e^{nu(y)}e^{nu(x)}\dy \dx  \nonumber\\
    &\leq 0.\label{J3_negative}
\end{align}
Therefore, using \eqref{J1_zero}, \eqref{J2_zero} and \eqref{J3_negative} altogether in \eqref{Double_decom}, we get \eqref{Double_zero}.  Since $Q_0 e^{nv}\in L^1(\R^n)$,  we  can find a sequence   $R_k\ra \infty$  such that 
\begin{equation}\label{Boundary_zero}
    \lim_{k\ra \infty}  R_k\int_{\pa B_{R_k}} Q_0(x) e^{nv(x)}\dS(x)  =0.
\end{equation}
 Therefore, taking $k\ra \infty$ in \eqref{Inte} along with \eqref{Double_zero} and \eqref{Boundary_zero} yields 
 \begin{align}
     \frac{1}{2\ga_n}\left(\int_{\R^n} Q_0(y) e^{nv(y)} \dy\right) \left(\int_{\R^n} Q_0(x) e^{nv(x)}\dx\right) & - \int_{\R^n} Q_0(x) e^{nv(x)}\dx \nonumber\\
     & \geq \frac{1}{n}\limsup_{k\ra \infty}\int_{B_{R_k}}(x\cdot \nabla Q_0(x))e^{nv(x)}\dx.
 \end{align}
Hence,
 \begin{equation}
     \frac{\La}{2\ga_n}(\La-2\ga_n)\geq \frac{1}{n}\limsup_{k\ra \infty}\int_{B_{R_k}}(x\cdot \nabla Q_0(x))e^{nv(x)}\dx.
 \end{equation}
 The proof is completed.
\end{proof}

\begin{proposition}\label{Pohozaev_equal}
	Let $u,v$ and $Q_0$ be as in Proposition \ref{Pohozaev}. Assume that $Q\geq 1$. Then there exist  $R_k\ra \infty$ such that 
	\begin{equation}\label{eq:pohozaev1}
		\frac{\La}{2\ga_n}(\La-2\ga_n)= \frac{1}{n}\limsup_{k\ra\infty}\int_{B_{R_k}}(x\cdot \nabla Q_0(x)) e^{nv(x)}\dx.
	\end{equation}
\end{proposition}
\begin{proof}
	By assumption, we have $Q\geq 1.$ Since $\La_*=\int_{\R^n} |x|^{n\al} Q (x)e^{nu(x)}\dx<+\infty,$ we get $$\La=\int_{\R^n} |x|^{n\al} e^{nu(x)}\dx\leq \La_*<+\infty.$$
	Moreover, using Corollary \ref{lamda_al}, we get $\ep>0$ such that
	\begin{equation}\label{lam:2alpha}
		(1+\al)+\ep<\be.
	\end{equation}
	To proceed further, we use similar methods as in Proposition \ref{Pohozaev}. Indeed, following the same notations, we claim that
	\begin{align}\label{double_zero}
		\limsup_{R\ra \infty}\int_{B_R}\int_{B_R^c} F(x,y)\dy\dx &=  0,
	\end{align}
	where $Q_0(z):=|z|^{n\al}e^{n(u-v)(z)}$ and $$F(x,y)=\frac{(x+y)\cdot (x-y)}{|x-y|^2}Q_0(y)Q_0(x)e^{nv(y)}e^{nv(x)}=\frac{(x+y)\cdot (x-y)}{|x-y|^2}|y|^{n\al}|x|^{n\al}e^{nu(y)}e^{nu(x)}.$$ 
	Observe that from \eqref{J1_zero} and \eqref{J2_zero}, we already have 
	\begin{equation}\label{J1J2_zero}
		J_1\ra 0\;\text{and}\;J_2\ra 0\quad \text{as}\;R\ra \infty.
	\end{equation}
	Now, we claim that 
	\begin{equation}\label{claimJ3}
		J_3=\int_{B_R\setminus B_{\frac{R}{2}}}\int_{B_{2R}\setminus B_R} F(x,y)\dy\dx\ra 0\quad\text{as}\;R\to \infty.
	\end{equation}
	Let $f(z):= |z|^{n\al} Q(z)e^{nu(z)}.$ Using Lemma \ref{u_weakBound}, we get large $R>0$ such that if $|x|\geq R,$
	\begin{align}
		u(x) & \leq \left(-\be+o(1)\right)\log(|x|)+ \frac{1}{\ga_n}\int_{B_1(x)}\log\left(\frac{1}{|x-z|}\right)f(z)\dz,\nonumber\\
		& =\left(-\be+o(1)\right)\log(|x|)+ \frac{||f||_{L^1\left(B_{R'}^c\right)}}{\ga_n}\int_{B_{R'}^c}\log\left(\frac{1}{|x-z|}\right) \frac{f(z)\chi_{\{|x-z|<1\}}}{||f||_{L^1\left(B_{R'}^c\right)}}\dz,\label{u_up}
	\end{align}
	where $R'=\frac{R}{4}.$  
	We can choose $R$ (sufficiently large) in such a way that  $\de:=\frac{n||f||_{L^1\left(B_{R'}^c\right)}}{\ga_n}<\frac12.$ Now, applying Jensen's inequality in \eqref{u_up} with respect to the measure $\d\mu(z):=\frac{f(z)}{||f||_{L^1\left(B_{R'}^c\right)}}\dz$, we obtain
	\begin{align}
		|x|^{n\al}e^{nu(x)} &\leq \frac{1}{|x|^{n\left(\be-\al\right)-no(1)}}\int_{B_{R'}^c} \left(1+\frac{1}{|x-z|^\de}\right)\d\mu(z)\nonumber\\
		&= \frac{1}{|x|^{n\kappa}}\int_{B_{R'}^c} \left(1+\frac{1}{|x-z|^\de}\right)\frac{f(z)}{||f||_{L^1(B_{R'}^c)}}\dz, 
	\end{align}
	where $\kappa=\be-\al-o(1).$ Thus for each $R<|y|<2R,$ we get 
	\begin{align}
		&\int_{B_R\setminus B_{\frac{R}{2}}}\frac{|x+y|}{|x-y|} |x|^{n\al}e^{nu(x)}\dx\nonumber\\
		& \leq \int_{B_R\setminus B_{\frac{R}{2}}}\frac{|x+y|}{|x-y|}\frac{1}{|x|^{n\kappa}}\int_{B_{R'}^c} \left(1+\frac{1}{|x-z|^\de}\right)\frac{f(z)}{||f||_{L^1\left(B_{R'}^c\right)}}\dz\dx\nonumber\\
		& \leq \frac{C}{R^{n\kappa-1}}\int_{B_R\setminus B_{\frac{R}{2}}}\frac{1}{|x-y|}\int_{B_{R'}^c} \left(1+\frac{1}{|x-z|^\de}\right)\frac{f(z)}{||f||_{L^1\left(B_{R'}^c\right)}}\dz\dx\nonumber\\
		&=\frac{C}{R^{n\kappa-1}}\frac{1}{||f||_{L^1\left(B_{R'}^c\right)}}\int_{B_{R'}^c} f(z)\left[\int_{B_R\setminus B_{\frac{R}{2}}}\frac{1}{|x-y|}\left(1+\frac{1}{|x-z|^\de}\right)\dx\right]\dz\nonumber\\
		&\leq \frac{CR^{n-1}}{R^{n\kappa-1}} \frac{1}{||f||_{L^1\left(B_{R'}^c\right)}}\int_{B_{R'}^c} f(z)\dz\nonumber\\
		&\leq\frac{C}{R^{n(\kappa-1)}}. \label{Bound}
	\end{align}
	Together with  \eqref{lam:2alpha} we  conclude that 	\begin{equation}
		\int_{B_R\setminus B_{\frac{R}{2}}}\frac{|x+y|}{|x-y|} |x|^{n\al}e^{nu(x)}\dx\leq C,\;\text{for each}\;R<|y|<2R.
	\end{equation}
	Now, using the fact that $\La<+\infty$ and the above estimate, we deduce
	\begin{align}
		|J_3| &\leq \int_{B_{2R}\setminus B_R}\int_{B_R\setminus B_{\frac{R}{2}}} |F(x,y)|\dx\dy\nonumber\\
		&=\int_{B_{2R}\setminus B_R}\int_{B_R\setminus B_{\frac{R}{2}}}\frac{|(x+y)\cdot (x-y)|}{|x-y|^2}|y|^{n\al}|x|^{n\al}e^{nu(y)}e^{nu(x)}\dx\dy\nonumber\\
		& \leq \int_{B_{2R}\setminus B_R} |y|^{n\al}e^{nu(y)}\left(\int_{B_R\setminus B_{\frac{R}{2}}}\frac{|x+y|}{|x-y|} |x|^{n\al}e^{nu(x)}\dx\right)\dy\nonumber\\
		&\leq C\int_{B_{2R}\setminus B_R} |y|^{n\al}e^{nu(y)}\dy\ra 0\;\quad\text{as}\;R\ra\infty.
	\end{align}
	Therefore, we proved our claim \eqref{claimJ3}, which essentially implies \eqref{double_zero}. Moreover, as in  \eqref{Boundary_zero}, we get a sequence   $R_k\ra \infty$ such that 
	\begin{equation}
		\lim_{k\ra \infty}  R_k\int_{\pa B_{R_k}} Q_0(x) e^{nv(x)}\dS(x)  =0.
	\end{equation}
	Hence, approaching along the same line as in the proof of Proposition \ref{Pohozaev}, we conclude 
	$$\frac{\La}{2\ga_n}(\La-2\ga_n)= \frac{1}{n}\lim_{k\ra\infty}\int_{B_{R_k}}(x\cdot \nabla Q_0(x)) e^{nv(x)}\dx.$$
	This finishes the proof.
\end{proof}

\section{Proofs of Theorems \ref{Bols_up} and \ref{Bols_low} }\label{Sec:proof_Bol}

In this section, we turn to the proof of Bol’s inequalities, beginning with Theorem \ref{Bols_up}.

\begin{proof}[Proof of Theorem \ref{Bols_up}]
   If $\La=+\infty,$ the statement is obvious. Thus we proceed with the assumption that $\La<+ \infty.$  Let $v$ be as defined in \eqref{v_standard}. 
 Define $h:=u-v.$ We claim that $\De h\geq 0$ in $\R^n.$ Indeed, if $n=2,$ then we have
\begin{equation}\label{De_h2}
    \De h (x)=(1-Q(x))|x|^{2\al} e^{2u(x)}.
\end{equation}
On the other hand, if $n\geq 3,$ we get 
\begin{equation}\label{De_h3}
    \De h (x)=\frac{n-2}{\ga_n}\int_{\R^n}\frac{(1-Q(y))}{|x-y|^2}|y|^{n\al}e^{nu(y)}\dy.
\end{equation}
In either case, it follows that  $\De h\geq 0$, as by our hypothesis $Q\leq 1.$
Let $R_k\to\infty$ be as in Proposition \ref{Pohozaev} so that 
\begin{equation}\label{Poho}
    \frac{\La}{2\ga_n}(\La-2\ga_n)\geq \frac{1}{n}\limsup_{k\ra \infty}\int_{B_{R_k}}(x\cdot \nabla Q_0(x)) e^{nv(x)}\dx,
\end{equation}
where $\La=\int_{\R^n} |x|^{n\al} e^{nu(x)}\dx$ and $Q_0(x):=|x|^{n\al}e^{nh(x)}.$  
Observe that
\begin{align}
    \int_{B_{R_k}}(x\cdot \nabla Q_0(x)) e^{nv(x)}\dx &= \int_{B_{R_k}}\left[x\cdot \nabla (|x|^{n\al}e^{nh(x)})\right] e^{nv(x)}\dx\nonumber\\
    &= n\al \int_{B_{R_k}}|x|^{n\al} e^{nu(x)}\dx + n\int_{B_{R_k}}(x\cdot \nabla h(x))|x|^{n\al} e^{nu(x)}\dx\label{grad_Q0}.
\end{align}
Moreover, $v$ is radial as $u$ is so and hence, $h$ is radial. Hence,  $\Delta h\geq0$ would imply $x\cdot\nabla h(x)\geq 0$,  which leads to 
\begin{align}
    \int_{B_{R_k}}(x\cdot \nabla h(x))|x|^{n\al} e^{nu(x)}\dx\ge0. 
\label{coarea}
\end{align}
Thus from \eqref{grad_Q0}, we get
\begin{align}\label{xQk}
	\int_{B_{R_k}}(x\cdot \nabla Q_0(x)) e^{nv(x)}\dx
	&\geq n\al \int_{B_{R_k}}|x|^{n\al} e^{nu(x)}\dx.
\end{align}
 Hence, using \eqref{xQk} in \eqref{Poho}, we obtain
 \begin{align}
 	\frac{\La}{2\ga_n}(\La-2\ga_n)\geq \limsup_{k\ra\infty}\left\{\al\int_{B_{R_k}}|x|^{n\al} e^{nu(x)}\dx \right\}=\al \La.
 \end{align}
Recalling that $\La_1=2\ga_n,$ we conclude 
\begin{equation}\label{Conclusion}
    \La\geq \La_1(1+\al),
\end{equation}
as desired. If $Q\equiv 1,$ then obviously we have the equality in \eqref{Conclusion} by \cite[Theorem 1.1]{Ali_Mancini}. Conversely, if equality happens in \eqref{Conclusion}, then we must have $\int_{B_R}(x\cdot \nabla h(x)) e^{nu(x)}\dx =0,$ which is possible only when 
\begin{align}
    \De h & \equiv 0\;\text{in}\;\R^n.
\end{align}
Therefore, using \eqref{De_h2} (for $n=2$) and \eqref{De_h3} (for $n\geq 3$), we conclude 
 that $Q\equiv 1\;\text{in}\;\R^n,$
and this ends the proof.
\end{proof}

We now establish Theorem \ref{Bols_low}, following an approach similar to that employed in the previous proof.
\begin{proof}[Proof of Theorem \ref{Bols_low}]
	 By hypothesis, we have $Q\geq 1.$ Moreover, $\La_*<+\infty.$ Therefore, 
	$$\La=\int_{\R^n}|x|^{n\al}e^{nu(x)}\dx\leq \int_{\R^n}|x|^{n\al}Q(x)e^{nu(x)}\dx=\La_*< +\infty.$$
	Let $h:=u-v,$ where $u$ is a normal solution of \eqref{Q-equation} and $v$ is given by \eqref{v_standard}. Now, for $x\in \R^n,$ 
	\begin{align*}
		\De h(x)= \begin{cases} 
			(1-Q(x))|x|^{2\al} e^{2u(x)}, & \text{if}\;n=2, \\
			\frac{n-2}{\ga_n}\int_{\R^n}\frac{(1-Q(y))}{|x-y|^2}|y|^{n\al}e^{nu(y)}\dy, & \text{if}\;n\geq 3.
		\end{cases}
	\end{align*}
	 Since $Q\geq 1,$ we get 
	 \begin{equation}\label{h_negative}
	 \De h\leq 0\;\text{in}\;\R^n.
	 \end{equation}
	  Now, using Proposition \ref{Pohozaev_equal},  we  get a
	sequence $R_k\ra \infty$ such that the below identity holds:
	\begin{equation}\label{eq:poho2}
		\frac{\La}{2\ga_n}(\La-2\ga_n)= \frac{1}{n}\limsup_{k\ra\infty}\int_{B_{R_k}}(x\cdot \nabla Q_0(x)) e^{nv(x)}\dx,
	\end{equation}
	where $Q_0(x):=|x|^{n\al}e^{nh(x)}.$ Next,
	proceeding in a similar manner as in \eqref{coarea} and using \eqref{h_negative}, we get
	\begin{align}
		\int_{B_{R_k}}(x\cdot \nabla Q_0(x)) e^{nv(x)}\dx
		&\leq n\al \int_{B_{R_k}}|x|^{n\al} e^{nu(x)}\dx.
	\end{align}
	Therefore, substituting the above in \eqref{eq:poho2} yields 
	$$\frac{\La}{2\ga_n}(\La-2\ga_n)\leq \limsup_{k\ra\infty} \left\{\al\int_{B_{R_k}}|x|^{n\al} e^{nu(x)}\dx\right\}=\al\La.$$
	Therefore, we must have 
	$$\La\leq \La_1(1+\al),$$ 
	as claimed in the statement. The equality case can be treated in the same manner as in the proof of Theorem \ref{Bols_up}, and the proof is therefore complete.
\end{proof}

\section{Proof of  Theorem \ref{existence_coro}}\label{Sec:existence}
 
We begin by noting that solutions to the singular equation  \eqref{Q1_equation} satisfy a Pohozaev identity, stated below, which is analogous to that derived in Proposition \ref{Pohozaev_equal}, see also \cite{Ali_Martinazzi}. The proof follows by adapting arguments similar to those used in Proposition \ref{Pohozaev_equal}; for brevity, we omit the detailed proof.

\begin{lemma}\label{Poho:singular}
	Let $u$ be a normal solution to \eqref{Q1_equation}. Then 
	\begin{align}\frac{\La_*}{2\ga_n}(\La_*-2\ga_n)= \al\int_{\R^n}|x|^{n\al} e^{nu(x)}\dx. \label{Poho_negative}\end{align}
\end{lemma}

Now, we give a proof of the necessary condition of Theorem \ref{existence_coro}.

\begin{proof}[Proof of Theorem \ref{existence_coro}]  (Necessary condition)
	Since, $\al<0,$ \eqref{Poho_negative} implies $\La_*<2\ga_n=\La_1.$
	Moreover, using $$\al\int_{\R^n}|x|^{n\al} e^{nu(x)}\dx>\al \int_{\R^n}(1+|x|^{n\al}) e^{nu(x)}\dx=\al \La_*,$$ along with \eqref{Poho_negative}, we obtain $\La_*>\La_1(1+\al).$ Therefore, we have
\begin{equation}\label{low_1}
	\La_1(1+\al)<\La_*<\La_1.
\end{equation}
We rewrite \eqref{Q1_equation} as 
\begin{align}
	\Dn u &= (1+|x|^{n\al}) e^{nu}
	 = Q |x|^{n\al}e^{nu},\;\quad 
	 Q(x):=(1+|x|^{-n\al}).
\end{align}
 Then $Q\geq1$, and therefore by Theorem \ref{Bols_low}  we get $\int_{\R^n}|x|^{n\al} e^{nu(x)}\dx<\La_1(1+\al).$ Hence, from \eqref{Poho_negative}, we deduce
\begin{align}
	\frac{\La_*}{2\ga_n}(\La_*-2\ga_n) &> \La_1 \al(1+\al),\nonumber\\
	\text{i.e.,}\; (\La_*+\La_1\al)(\La_* &-\La_1-\La_1\al) >0.
\end{align}
Since $\La_*-\La_1-\La_1\al>0$ in view of \eqref{low_1}, we must have $\La_*>-\La_1\al.$ Thus, combining with \eqref{low_1}, we conclude
$$\La_1\max\{-\al,1+\al\}<\La_*<\La_1,$$
as desired in the statement. 
\end{proof}

\medskip

We now show that $\La_1\max\{-\al,1+\al\}<\La_*<\La_1$ is also a sufficient condition for the existence of a radial normal solution to \eqref{Q1_equation}. To this end, we fix    a cut-off function $\vp$ given by   \begin{align*} \vp(t):=\left\{\begin{array}{ll}0\quad&\text{for }0\leq |t|< 1,\\  |t|-1\quad&\text{for }1\leq |t|<2,\\1\quad&\text{for }|t|\geq2,\end{array}  \right.   \end{align*} and for  $\ep>0,\delta>0$ small, we set   $$\vp_\delta(x):=\vp\left(\frac{ |x|}{\delta}\right), \quad \psi_\ep(x):=1-\vp(\ep |x|).$$
	\begin{proposition}\label{prop-51} Let $p\in (0,1)$ be fixed. For every $\ep>0,\delta>0$ and $\Lambda_*\in (0,\Lambda_1)$, there exists a radial normal solution to $$(-\Delta)^\frac n2 u=\left(1+e^{-npu(0)}|x|^{n\alpha}\vp_\delta\right)\psi_\ep e^{nu},\quad \Lambda_*=\int_{\R^n}\left(1+e^{-npu(0)}|x|^{n\alpha}\vp_\delta\right)\psi_\ep e^{nu} \dx.$$
\end{proposition}
	\begin{proof} The proof will be based on a fixed point argument on the space $\X$  defined by \begin{equation}\label{X}
		\X :=\left\{v\in \mathcal{C}^0_{\rm rad}(\R^n): ||v||_{\X}:=\sup_{x\in \R^n}\frac{|v(x)|}{\log(|x|+2)}<\infty\right\}.
	\end{equation} For each $v\in \X$, we fix $c_v\in\R$ such that $$\int_{\R^n} (1+e^{-np[v(0)+c_v]}|x|^{n\alpha}\vp_\delta)\psi_\ep e^{n[v(x)+c_v]}\dx=\Lambda_*.$$  Existence of such constant $c_v$ follows from the fact that the function $$G(t):=\int_{\R^n} (1+e^{-np[v(0)+t]}|x|^{n\alpha}\vp_\delta)\psi_\ep e^{n[v(x)+t]}\dx,\quad t\in\R,$$ is monotone increasing as $p\in(0,1)$, and $$G(t)\xrightarrow{t\to\infty}\infty,\quad G(t)\xrightarrow{t\to-\infty}0.$$  Moreover, $v\mapsto c_v$ is continuous on $\X$. One can show that the map $v\to \bar v$ is compact, where $\bar v\in \X$ is a radial normal solution to $$(-\Delta)^\frac n2 \bar v=(1+e^{-np[v(0)+c_v]}|x|^{n\alpha}\vp_\delta)\psi_\ep e^{n[v(x)+c_v]}.$$ As $\Lambda_*<\Lambda_1$, one easily shows that  there exists a fixed point $v$, and consequently $u=v+c_v$ is a desired solution. 
	\end{proof}
	
	\begin{remark} For the radial normal solutions $u=u_{\ep,\delta}$, obtained in the above proposition, we have \begin{align}  \label{cond-normal}\int_{B_R}|\nabla^\ell u_{\ep,\delta}(x)|\dx\leq CR^{n-\ell}\quad\text{for every }R>0, \, 1\leq\ell\leq n-1,\end{align} where the constant $C $ is independent of $\ep>0, \delta>0$. 
	\end{remark}
	
	For each fixed $\delta>0,$ we now study the behavior of the family of solutions $\{u_\ep\}$ constructed in Proposition \ref{prop-51} as $\ep\to0$. In the following,  $p\in (0,1)$ and $\alpha\in (-1,0)$ are fixed with $p+\alpha>0$.  For each  $u_\ep$,  we have the following Pohozaev identity: 
	\begin{align}  \frac{\Lambda_*(\Lambda_*-\Lambda_1)}{\Lambda_1}= \alpha e^{-npu_\ep(0)}  \int_{\R^n} & \vp_\delta |x|^{n\alpha} \psi_\ep e^{nu_\ep} \dx 
		 + \frac{e^{-npu_\ep(0)}}{n}\int_{B_{2\de}\setminus B_\de} (x\cdot \nabla \phi_\de(x))|x|^{n\al}\psi_\ep e^{nu_\ep}\dx\\
		& +\frac{1}{n}\int_{B_{\frac{2}{\ep}}\setminus B_{\frac{1}{\ep}}} (x\cdot \nabla \psi_\ep(x)) (1+e^{-npu_\ep(0)}|x|^{n\al}\phi_\de)e^{nu_\ep}\dx.  \label{poh-2}
	\end{align}
	
	\begin{lemma} \label{lem-4.4}If $u_\ep(0)\to -\infty$, then $$ e^{-np u_\ep(0)} \int_{\R^n}|x\cdot\nabla\vp_\delta(x)||x|^{n\alpha}\psi_\ep e^{nu_\ep}\dx\to0.  
	$$ \end{lemma}
	
	\begin{proof} Notice that  $\nabla \vp_\delta$ is  supported on $B_{2\delta}\setminus B_\delta$, and as $u_\ep$ is monotone decreasing and $p<1$, we see that $$ |x\cdot\nabla\vp_\delta(x)|  e^{nu_\ep(x)-npu_\ep(0)}\leq 2 e^{n(1-p)u_\ep(0)} \to0.$$ Hence, $$e^{-np u_\ep(0)} \int_{\R^n}|x\cdot\nabla\vp_\delta(x)|x|^{n\alpha}\psi_\ep e^{nu_\ep}\dx\leq 2e^{n(1-p)u_\ep(0)} \int_{B_{2\delta}\setminus B_\delta}|x|^{n\alpha}\dx\to0. $$ 
	\end{proof}

	\begin{proposition} \label{propo-2}Let $\La_1\max\{-\al,1+\al\}<\La_*<\La_1$ be fixed. Then there exists $C>>1$ such that $u_\ep(0)\geq -C$ .
	\end{proposition}
	
	\begin{proof} Assume by contradiction that for some subsequence $\ep\to0$ (we still denote it with the same notation $\ep$) we have $u_\ep(0)\to-\infty$. Set $$\hat u_\ep(x):=u_\ep(r_\ep x)-u_\ep(0),\quad r_\ep:=e^{-\frac{1-p}{1+\alpha}u_\ep(0) }\to+\infty.$$ Then $\hat u_\ep$ is a radial normal solution to $$ (-\Delta)^\frac n2 \hat u_\ep=\hat f_\ep,\quad \hat f_\ep(x):=\left(  r_\ep^{q} +|x|^{n\alpha} \hat\vp_\delta \right)\hat \psi_\ep e^{n\hat u_\ep},\quad q:=-\frac{p+\alpha}{1-p}<0,$$ where $\hat \vp_\delta (x):=\vp_\delta(r_\ep x) $ and $\hat \psi_\ep(x):=\psi_\ep(r_\ep x)$. As $r_\ep^q\to0$, we get that  for $\ep>0$ small, $$\hat f_\ep(x)\leq 1+|x|^{n\alpha} \quad\text{in }\R^n,$$ and hence, by standard  elliptic   estimates,  up to a subsequence, \begin{align}\label{conv}\hat u_\ep\to \hat u\quad\text{in }C^0_{\rm loc}(\R^n). \end{align} Concerning the equation satisfied by $\hat u$, we consider the following three cases: 
	
	\noindent \textbf{Case 1:} $\ep r_\ep\to  \infty $.
	
We easily arrive at a contradiction, since  from the definition of $\hat \psi_\ep$,  one obtains  \begin{align}  \Lambda_*=\int_{\R^n}\hat f_\ep \dx\leq \int_{|x|\leq \frac{2}{\ep r_\ep}}(1+|x|^{n\alpha})\dx\to0. \end{align}

\noindent \textbf{Case 2:} $\ep r_\ep\to a\in (0,\infty) $. 

In this case, it follows immediately  that $\hat u$ is a radial normal solution to $$(-\Delta)^\frac n2 \hat u=|x|^{n\alpha}\psi_a e^{n\hat u}.$$ Moreover, as $\psi_\ep$ is supported in $B_{2/a+1}$ for $\ep>0$ small enough, from \eqref{conv} we get that $$\Lambda_*=\int_{\R^n}\hat f_\ep \dx\xrightarrow{\ep\to0} \int_{\R^n}|x|^{n\alpha}\psi_a e^{n\hat u}\dx.$$ Since $x\cdot\nabla\psi_a(x)\leq0$ in $\R^n$, the Pohozaev identity \eqref{poh-2} for $\hat u$ leads to $$\frac{\Lambda_*(\Lambda_*-\Lambda_1)}{\Lambda_1}< \alpha\Lambda_*\quad\Longrightarrow \Lambda_*<\Lambda_1(1+\alpha),$$  in contradiction with the assumption on $\Lambda_*$. 

\noindent \textbf{Case 3:} $\ep r_\ep\to0$. 
	
	In this case, the limit function $\hat u$ is a radial  solution to  \eqref{Q-equation} with $Q\equiv 1$. 
	Moreover, as each $\hat u_\ep$ satisfies \eqref{cond-normal}, so does  $\hat u$, and therefore, it is a normal solution.  Consequently, \begin{align} \label{con-1} \int_{\R^n}|x|^{n\alpha} e^{n\hat u} \dx=\Lambda_1(1+\alpha).  \end{align}
		Next, we claim that \begin{align} \label{vanish-1}\lim_{\ep\to 0}\int_{\R^n}  \hat \vp_\delta |x|^{n\alpha} \hat \psi_\ep e^{n\hat u_\ep}\dx=\Lambda_1(1+\alpha).\end{align} 
	We fix $\theta>0$ small such that $1+\alpha-4\theta>0$.  Let $R_1>>1$ be such that $$\int _{B_{R_1}}| x|^{n\alpha} e^{n\hat u} \dx\geq \Lambda_1(1+\alpha-\frac\theta2).$$ Then for $\ep>0$ small enough, we have $$\int _{B_{R_1}}| x|^{n\alpha} \hat\vp_\delta e^{n\hat u_\ep} \dx\geq \Lambda_1(1+\alpha-\frac 23\theta).$$   This leads to \begin{align}  \Delta\hat u_\ep (x) \leq  -\frac{n-2}{\gamma_n} \int_{B_{R_1}} \hat f_\ep(y)\frac{\dy}{|x-y|^2}\leq -2(n-2) (1+\alpha-\theta)\frac{1}{|x|^2}\quad\text{for }|x|\geq R_2>>R_1.\end{align} Therefore, for $R\geq R_3>>R_2$ 
	\begin{align} \int_{B_R} \Delta \hat u_\ep \dx\leq \int_{B_R\setminus B_{R_2}} \Delta\hat u_\ep \dx & \leq -2 (1+\alpha-\theta)|\S^{n-1}| (R^{n-2}-R_2^{n-2})\\
		& \leq -2 (1+\alpha-2\theta)|\S^{n-1}| R^{n-2}. \end{align} Finally, by \eqref{rad-relation} $$\hat u_\ep(x)\leq -2(1+\alpha-2\theta)\log |x|+C_0,\quad \text{for }|x|\geq R_3,$$ for some constant $C_0=C_0(R_3)$, which can be made independent of $\ep\to0$. Recalling  that $1+\alpha-4\theta>0$, for  $R_4>>R_3$, we obtain  \begin{align} \int_{|x|\geq R_4}\hat \vp_\delta |x|^{n\alpha} \hat \psi_\ep e^{n\hat u_\ep}\dx  \leq e^{nC_0}\int_{|x|\geq R_4} |x|^{-n(2+\alpha-4\theta)}\dx\leq CR_4^{-n(1+\alpha-4\theta)}\xrightarrow{R_4\to\infty}0 ,\end{align} which proves \eqref{vanish-1}. 
	
Together with Lemma \ref{lem-4.4},  \eqref{vanish-1} and  the fact that $x\cdot\nabla\psi_\ep(x)\leq0$,    we deduce from \eqref{poh-2} $$\frac{\Lambda_*(\Lambda_*-\Lambda_1)}{\Lambda_1}\leq \alpha [\Lambda_1(1+\alpha)+o(1)] +o(1),$$ which gives   $$(\Lambda_*+\alpha\Lambda_1)(\Lambda_*-(1+\alpha)\Lambda_1)\leq0,$$ a contradiction to the  assumption that $\Lambda_*>\Lambda_1\max\{-\alpha, 1+\alpha \}$. 
	\end{proof}
	
	\begin{remark} For a fixed $\Lambda_*\in (0,\Lambda_1)$, if $u_\ep$ is a radial normal solution as constructed in Proposition \ref{prop-51} without the term $|x|^{n\alpha}\vp_\delta$ (that is, $(-\Delta)^\frac n2 u_\ep=\psi_\ep e^{nu_\ep}$), then necessarily $u_\ep(0)\to-\infty$. In fact, with the same notations as in the proof of Proposition \ref{propo-2}, we must have   $\ep r_\ep\approx 1 $, and   $\hat u_\ep\to u$, where $u$ is a radial normal solution of  $(-\Delta)^\frac n2 u=\psi_a e^{nu}$ for some $a>0$.
	
	\end{remark}
	
	\begin{lemma} \label{lem-normal} For each $\delta>0$ fixed, up to a subsequence, $u_\ep\to u$ in $C^0_{\rm loc}(\R^n)$, where $u$ is a radially symmetric normal solution to  \begin{align}\label{eq-normal-2}(-\Delta)^\frac n2 u=(1+e^{-np u(0)}\vp_\delta|x|^{n\alpha})e^{nu},\end{align} satisfying \begin{align} \Lambda_*=\int_{\R^n} (1+e^{-np u(0)}\vp_\delta|x|^{n\alpha})e^{nu}\dx.\label{curv}\end{align} \end{lemma}
	\begin{proof} Since $\vp_\delta\equiv0$ in $B_\delta$, and $\Lambda_*<\Lambda_1$, we must have $u_\ep(0)\leq C$. On the contrary,  there will be \emph{spherical bubble} at the origin, which is not possible as $\Lambda_*<\Lambda_1$. 
	
By Proposition \ref{propo-2} we conclude that $|u_\ep(0)|\leq C$, and hence, by standard elliptic estimates, up to a subsequence, $u_\ep\to u$ in $C^0_{\rm loc}(\R^n)$, where $u$ is a normal solution to \eqref{eq-normal-2} with $$\hat \Lambda :=\int_{\R^n} (1+e^{-np u(0)}\vp_\delta|x|^{n\alpha})e^{nu}\dx<\infty.$$ It remains to show that $\Lambda_*=\hat \Lambda$, which is equivalent to \begin{align}  \label{vanish-2} \lim_{R\to\infty} \lim_{\ep\to0}\int_{B_R^c} (1+e^{-np u_\ep(0)}\vp_\delta|x|^{n\alpha})\psi_\ep e^{nu_\ep}\dx=0.\end{align} In order to prove \eqref{vanish-2}, first, we notice that  by Lemma \ref{u_UpLow},
 $$u(x)\geq -\beta\log|x|-C\quad \text{for }|x|\geq R,\quad \beta:=\frac{\hat\Lambda}{\gamma_n}=2\frac{\hat\Lambda}{\Lambda_1},$$ and by the finiteness of $\hat\Lambda$, we necessarily have that $\beta>1$, that is $\hat\Lambda>\frac12\Lambda_1$. Then one can proceed as in the proof of \eqref{vanish-1} to conclude \eqref{vanish-2}. 
	\end{proof}
	
	\begin{proof}[Proof of Theorem \ref{existence_coro}] (Existence part) For each $\delta>0$ small, let $u=u_\delta$  (not to confuse with the previous notation $u_\ep$) be a normal solution to \eqref{eq-normal-2}-\eqref{curv} as given by Lemma \ref{lem-normal}. We now show that, up to a subsequence, $u_\delta\to u$ in $C^0_{\rm loc}(\R^n)$, where  $u$ is a radial normal solution to \eqref{Q1_equation} (up to a scaling) as desired. 
	
	\noindent\textbf{Step 1:} $u_\delta(0)\leq C$. 
	
	Assume by contradiction that for some subsequence (still denoted by the same notation) we have $u_\delta(0)\to+\infty$. Then setting $$v_\delta(x):=u_\delta(r_\delta x)-u_\delta(0),\quad r_\delta:=e^{-u_\delta(0)}\to0,$$ we see that $v_\delta$ is a radial normal solution to $$(-\Delta)^\frac n2 v_\delta=(1+ r_\delta ^{n(p+\alpha)}|x|^{n\alpha}\hat\vp_\delta )e^{n v_\delta},\quad \Lambda_*=\int_{\R^n}(1+ r_\delta ^{n(p+\alpha)}|x|^{n\alpha}\hat\vp_\delta )e^{n v_\delta}\dx,$$ where $\hat \vp_\delta(x):=\vp_\delta(r_\delta x)$. Since $p+\alpha>0$, up to a subsequence, $v_\delta\to v$ in $C^0_{\rm loc}(\R^n)$, where $v$ is a radial normal solution to $$ (-\Delta)^\frac n2 v=e^{nv}\quad\text{in }\R^n,\quad \int_{\R^n} e^{nv}\dx\leq\Lambda^*.$$  This contradicts to the fact that $$\Lambda_1=\int_{\R^n} e^{nv}\dx\leq\Lambda^*<\Lambda_1.$$
	
	\noindent\textbf{Step 2:} $u_\delta(0)\geq -C$. 
	
	Assume by contradiction that $u_\delta(0)\to-\infty$ as $\delta\to0$. This would lead to a contradiction as in the proof of Proposition \ref{propo-2}. 
	
	Thus, $|u_\delta(0)|\leq C$, and hence, up to a subsequence, $u_\delta\to u$ in $C^0_{\rm loc}(\R^n)$, where $u$ is a radial normal solution to  $$(-\Delta)^\frac n2 u=(1+e^{-np u(0)}|x|^{n\alpha})e^{nu}, \quad \hat \Lambda:=\int_{\R^n} (1+e^{-np u(0)} |x|^{n\alpha})e^{nu}\dx\leq\Lambda_*.$$ Then by Lemma \ref{u_UpLow}  and Corollary \ref{lamda_al} we necessarily have that $\hat\Lambda>\frac12\Lambda_1$, and   following the proof of \eqref{vanish-1}, one gets $\hat \Lambda=\Lambda_*$. Finally,  the function $$\tilde u(x)=u(\lambda x)+\log\lambda,\quad \lambda:=e^{\frac p\alpha u(0)},$$ is a desired radial normal solution to \eqref{Q1_equation}.
	\end{proof}

\section{Proof of Theorem  \ref{f_existence}}\label{Sec:existence_f}

 Let $f\in C^0_{\text{rad}}(\R^n)\cap L^1(\R^n)$ be a given non-negative function. For  $\ep>0$,  let $u_\ep$ be a radial normal solution of 
\begin{equation}\label{u_ep_eq}
	\left\{\begin{aligned}
		\Dn u_\ep &= (1+f)|x|^{n\al} e^{nu_\ep-\ep|x|^2}\quad \text{in}\; \R^n,\\
		\int_{\R^n}(1+ &f(x))|x|^{n\al} e^{nu_\ep(x)-\ep|x|^2}\dx< \infty.
	\end{aligned}\right.
\end{equation}
We define 
\begin{equation}\label{v_eps}
	v_\ep(x)=\frac{1}{\ga_n}\int_{\R^n}\log \left(\frac{1+|y|}{|x-y|}\right)|y|^{n\al} e^{nu_\ep(y)-\ep|y|^2}\dy.
\end{equation}

\begin{proposition}\label{Poho_ep}
	For a given $\ep>0$, let $u_\ep$   and $v_\ep$ be as   above. Define
	\begin{equation}
		\tilde{\La}_\ep:=\int_{\R^n}|x|^{n\al}e^{nu_\ep(x)-\ep|x|^2}\dx, \quad Q_0(x):=|x|^{n\al}e^{nu_\ep(x)-nv_\ep(x)-\ep|x|^2}.
	\end{equation}
	Then there exist $R_k\ra \infty$ such that the following  identity holds:
	$$\frac{\tilde{\La}_\ep}{\La_1}(\tilde{\La}_\ep-\La_1)= \frac{1}{n} \lim_{k\ra\infty}\int_{B_{R_k}}(x\cdot \nabla Q_0(x)) e^{nv_\ep(x)}\dx.$$
\end{proposition}
\begin{proof}
	The proof is similar to Proposition \ref{Pohozaev_equal}.
\end{proof}

As an immediate consequence of the above identity, we obtain: 
\begin{lemma}\label{Bols_ep}
	Suppose $f\in C^0_{\rm{rad}}(\R^n)\cap L^1(\R^n)$ is non-negative.
	Then we have 
	$$\tilde{\La}_\ep\leq \La_1(1+\al). $$ 
\end{lemma}
\begin{proof}
	Setting $h_\ep=u_\ep-v_\ep$,   we get $\De h_\ep\leq 0$ in $\R^n$ due to the non-negativity of $f.$ In particular, $$x\cdot\nabla h_\ep(x)\leq0\quad\text{in }\R^n.$$ Therefore, $$x\cdot\nabla Q_0(x)\leq n\alpha Q_0(x)-2\ep |x|^2 Q_0(x),$$ and hence, by Proposition  \ref{Poho_ep}, $$\tilde{\La}_\ep\leq \La_1(1+\al). $$
	\end{proof}

\begin{proof}[Proof of Theorem \ref{f_existence}]
	 Let $\rho\in\R$ be given. For $\ep>0,$ consider the following perturbed equation:
	\begin{equation}\label{Q-eq_pert2}
		\left\{\begin{aligned}
			\Dn u_\ep &= (1+f)|x|^{n\al} e^{nu_\ep-\ep|x|^2}\quad \text{in}\; \R^n,\\
			u_\ep(0) &=\rho.
		\end{aligned}\right.
	\end{equation}
	Making similar arguments as in \cite[Theorem 6.1]{Ali_Martinazzi}, we get a radial normal solution $u_\ep$ of \eqref{Q-eq_pert2} satisfying $u_\ep(0)=\rho$ for each $\ep>0$.   Recall that from Lemma \ref{Bols_ep} \begin{align}\label{uni4}\tilde{\La}_\ep=\int_{\R^n} |x|^{n\al} e^{nu_\ep-\ep|x|^2}\dx\leq\Lambda_1(1+\alpha).\end{align} Moreover, as the upper bound of $\tilde\Lambda_\ep$ is independent of $\ep>0$ and $\rho\in\R$,  following   the proof of Lemma \ref{lem:enu_bound} below, we obtain  \begin{align}\label{uni5}|x|^{n\alpha} e^{nu_\ep-\ep|x|^2}\leq C\quad\text{for }|x|\geq 1,\end{align} for some constant $C>0$ independent of $\ep>0$ and $\rho\in\R$.  This leads to 
	\begin{align}
		\La_\ep :=\int_{\R^n} (1+f(x))|x|^{n\al} e^{nu_\ep-\ep|x|^2}\dx \leq \tilde\Lambda_\ep (1+\|f\|_{L^\infty(B_1)})+C\|f\|_{L^1(\R^n)}\leq C.
	\end{align}
It then follows that, up to a subsequence, $u_\ep\ra u_\rho$ in $\C^{0}_{\rm loc}(\R^n)$ as $\ep\ra 0$, where  $u_\rho$ is normal solution of \eqref{Q-equation} with $Q=1+f$ and $u_\rho(0)=\rho$.
We set 
	\begin{equation}\label{La_rho}
		\La(\rho):= \int_{\R^n} (1+f(x)) |x|^{n\al} e^{nu_\rho}\dx.
	\end{equation}

	\noi \textbf{Case 1:} $(\rho\rightarrow+\infty)$
	
	We set 
	\begin{equation}
		\eta_\rho(x):=u_\rho(r_\rho x)-u_\rho(0), \quad r_\rho:=e^{-\frac{1}{1+\alpha}u_\rho(0)}\xrightarrow{\rho\to\infty}0.
	\end{equation}
	Then,  as $f$ is continuous, up to a subsequence,  $\eta_\rho\rightarrow \eta$ in $\C^0_{\text{loc}}(\R^n)$, where   $\eta$ is a normal solution of 
	\begin{equation}
		\Dn \eta = (1+f(0))|x|^{n\al} e^{n\eta}\;\text{in}\;\R^n,\quad\int_{\R^n}(1+f(0))|x|^{n\al}e^{n\eta}\dx<\infty, 
	\end{equation} and in particular, we have $$\Lambda_1(1+\alpha)=\int_{\R^n}(1+f(0))|x|^{n\al}e^{n\eta}\dx=\lim_{R\to\infty}\lim_{\rho\to\infty}\int_{B_{Rr_\rho}}(1+f)|x|^{n\alpha}e^{nu_\rho}\dx.$$ Then necessarily \begin{align}\label{uni6}u_\rho\to-\infty\quad\text{in }C^0_{\text{loc}}(\R^n\setminus\{0\}).\end{align} On the contrary, up to a subsequence, we will have $u_\rho\to \hat u$ in $C^0_{\text{loc}}(\R^n\setminus\{0\})$, where $$\hat u(x)\geq 2(1+\alpha)\log\frac{1}{|x|}-C\quad\text{for }0<|x|\leq1,$$ a contradiction to $$\int_{B_1}|x|^{n\alpha }e^{n\hat u}\dx<\infty.$$ 
	
	Following the proof of \eqref{vanish-1}, we can show that $$\lim_{R\to\infty}\lim_{\rho\to\infty}\int_{B_{Rr_\rho}^c} |x|^{n\alpha}e^{nu_\rho}\dx=\lim_{R\to\infty}\lim_{\rho\to\infty}\int_{B_R^c} |x|^{n\alpha}e^{n\eta_\rho}\dx=0.$$ Together with \eqref{uni5}, \eqref{uni6} and the fact that $f\in C^0(\R^n)\cap L^1(\R^n)$, we conclude that $$\lim_{R\to\infty}\lim_{\rho\to\infty}\int_{B_{Rr_\rho}^c} f|x|^{n\alpha}e^{nu_\rho}\dx=0.$$ Thus $$\lim_{\rho\rightarrow+\infty} \La (\rho) =\La_1(1+\al).$$
	

	

\noi \textbf{Case 2:} $(\rho\rightarrow-\infty)$

It follows from \eqref{uni5} that \begin{align}  \lim_{\rho\to-\infty}\int_{\R^n}f|x|^{n\alpha} e^{nu_\rho}\dx=0.\label{uni7}  \end{align}
 Define
\begin{align}
	v_\rho(x) :=\frac{1}{\ga_n}\int_{\R^n}\log \left(\frac{1+|y|}{|x-y|}\right) |x|^{n\al}e^{nu_\rho(y)}\dy,\quad h_\rho :=u_\rho-v_\rho.
\end{align}
Then as in Proposition \ref{Pohozaev_equal}, we have 
	\begin{align}\label{P_hrho}
		\frac{\Tilde{\La}(\rho)}{\La_1}(& \Tilde{\La}(\rho) -\La_1) =\frac{1}{n}\limsup_{k\rightarrow+\infty} \int_{B_{R_k}}(x\cdot \nabla Q_0(x))e^{nv_\rho(x)}\dx ,\end{align} for some $R_k\to\infty$, where 
 $$\Tilde{\La}(\rho):=\int_{\R^n}|x|^{n\al}e^{nu_\rho(x)}\dx,\quad  Q_0(x):=|x|^{n\al}e^{nh_\rho(x)}.$$ 
 From the definition of $h_\rho$,  one gets
	\begin{align*}
		\De h_\rho(x)= \begin{cases} 
			-f(x)|x|^{2\al}e^{2u_\rho(x)}, & \text{if}\;n=2, \\
			-\frac{n-2}{\ga_n}\int_{\R^n}\frac{f(y)}{|x-y|^2}|y|^{n\al}e^{nu_\rho(y)}\dy, & \text{if}\;n\geq 3.
		\end{cases}
	\end{align*}
Therefore,  for $n=2$, 
	\begin{align}
		0\geq \int_{B_r}\De h_\rho(x)\dx=  -\int_{B_r} f(x)|x|^{2\al}e^{2u_\rho}\dx,  
	\end{align}  
	and for $n\geq 3,$ 
	\begin{align}
		0\geq\int_{B_r}\De h_\rho(x)\dx & = -\frac{n-2}{\ga_n} \int_{B_r} \int_{\R^n} \frac{f(y)}{|x-y|^2}|y|^{n\al}e^{nu_{\rho}(y)}\dy \dx\\
		&=- C \int_{\R^n} f(y) |y|^{n\al}e^{n u_\rho(y)} \left(\int_{B_r} \frac{1}{|x-y|^2} \dx \right)\dy\\
		&\geq - C r^{n-2} \int_{\R^n} f(y)|y|^{n\al} e^{n u_\rho(y)} \dy.
	\end{align} 
	Moreover, using integration by parts, we get   $$\int_{B_R} x\cdot\nabla h(x) |x|^{n\alpha} e^{nu_\rho}\dx=\int_{0}^{R} r^{n\al+1}e^{nu_\rho(r)} \int_{B_r} \De h_\rho(|x|)\dx\dr.$$ Therefore, by \eqref{uni4} and \eqref{uni7},   
	\begin{align}
		0 & \geq \limsup_{k\rightarrow+\infty}\int_{0}^{R_k} r^{n\al+1}e^{nu_\rho(r)} \int_{B_r} \De h_\rho(|x|)\dx\dr \\
		& \geq  -C \limsup_{k\rightarrow+\infty} \left\{\int_{0}^{R_k} r^{n\al+n-1} e^{n u_\rho(r)} \left(\int_{\R^n} f(y)|y|^{n\al} e^{n u_\rho(y)} \dy\right)\dr\right\}\\
		&=-C \left(\int_{\R^n} f(y)|y|^{n\al} e^{n u_\rho(y)} \dy\right) \left(\int_{\R^n}|x|^{n\al}e^{n u_\rho(x)}\dx\right) \\ &\xrightarrow{\rho\to-\infty} 0.
	\end{align}  Using this in  \eqref{P_hrho}, we deduce
 	\begin{equation}
		\Tilde{\La}(\rho)\rightarrow \La_1(1+\al)\;\quad \text{as}\;\rho\rightarrow-\infty,
	\end{equation}
	and again by \eqref{uni7},
	\begin{align}
		\La(\rho)=\int_{\R^n} (1+f(x))|x|^{n\al}e^{nu_\rho}\dx=\Tilde{\La}(\rho)+\int_{\R^n}f|x|^{n\al}e^{nu_\rho}\dx \rightarrow \La_1(1+\al)\;\text{as}\;\rho\rightarrow-\infty.
	\end{align} 
The proof is concluded. 
	\end{proof}
 
	\section{Proof of  Theorem  \ref{total-curvature}}\label{Sec:bound_proof}

	We now establish Theorem \ref{total-curvature}  with the help of Theorem \ref{Bols_low}.  To this end, we first require the following auxiliary lemma:

\begin{lemma}\label{lem:enu_bound}
	Let $u$ be a radial normal solution of \eqref{Q-equation} with $Q$ be as in Theorem \ref{total-curvature}. Then for $\al>0,$ there exists $C>0$ (not depending on $u$) such that
	$$|x|^{n\al}e^{nu(x)}\leq C\;\text{in} \;B_1^c.$$
\end{lemma}

\begin{proof}
	Set $\tilde u(x)=u+\alpha \log|x|.$ Since $u$ is decreasing, one gets 
	\begin{equation}\label{tilde_u_2}
		\tilde u'(x)\leq \frac{\alpha}{|x|}.
	\end{equation} 
	We claim that $\tilde u\leq C$ on $B_1^c.$ Let $R>0$ and $R_0\in[R,R+1]$ be such that $$\tilde u(R_0)=\min_{\bar B_{R+1}\setminus B_R}\tilde u. $$   
	Then by Theorem \ref{Bols_low}, \begin{align}
		\La_1(1+\al)\geq\int_{\R^n} |x|^{n\al}e^{nu}\dx=\int_{\R^n} e^{n\tilde u}\dx \geq e^{n\tilde u(R_0)}|B_{R+1}\setminus B_R|,
	\end{align} which yields $\tilde u(R_0)\leq C_0$ for some  constant $C_0=C_0(n,\alpha)>0$. Notice that for $$\frac12\leq R_1<R_2\leq R_1+2,$$  we have $$\tilde u(R_2)\leq \tilde u(R_1)+4\alpha,$$ thanks to \eqref{tilde_u_2}.  The lemma  follows immediately.  
\end{proof}

\begin{proof}[Proof of Theorem \ref{total-curvature}]
	Let $u$ be a radial normal solution of \eqref{Q-equation}. By assumption, we have $1\leq Q$ and hence, applying 
	Theorem \ref{Bols_low}, we get 
	\begin{equation}\label{bol_1}
		\int_{\R^n} |x|^{n\al} e^{nu(x)}\dx\leq \La_1(1+\al).
	\end{equation}
	Let $p\geq 0$ be as given in the statement. Rewriting $Q|x|^{n\al}e^{nu}=\widetilde{Q}|x|^{n\al+p}e^{nu},$ where $\widetilde{Q}=\frac{Q}{|x|^p},$ and using the assumption, one has $\widetilde{Q}\geq 1.$ Therefore, again  by  Theorem \ref{Bols_low}  (though the function $\tilde Q$ is not in $L^\infty (B_1)$, or not even in $L^1 (B_1)$ for $p\geq n$, we can still apply Theorem \ref{Bols_low} as the the quantity $\tilde Q |x|^{n\alpha+p}=Q|x|^{np}$  is in $L^q ( B_1)$ for some $q>1$), we obtain 
	\begin{equation}
		\label{bol_2}
		\int_{\R^n} |x|^{n\al+p} e^{nu(x)}\dx\leq \La_1\left(1+\frac{n\al+p}{n}\right).
	\end{equation}
	Now, 
	\begin{align}
		\int_{\R^n} (Q-M|x|^p)|x|^{n\al}e^{nu}\dx &=\int_{B_1} (Q-M|x|^p)|x|^{n\al}e^{nu}\dx+\underbrace{\int_{B_1^c} (Q-M|x|^p)|x|^{n\al}e^{nu}\dx}_{:=I}\\
		&\leq C \int_{B_1}|x|^{n\al}e^{nu} + I\\
		&\leq C \La_1(1+\al)+I.\label{QM}  
	\end{align}
	Next, we bound $I$. Observe that $\De u< 0$ in $\R^n\setminus \{0\}$. Thus divergence theorem yields
	\begin{equation}\label{rad-relation}
		u(r_1)-u(r_2)=\int_{r_2}^{r_1}\frac{1}{\om_{n-1}r^{n-1}}\int_{B_r}\De u(x)\dx\dr<0,\;\text{for}\;0<r_2<r_1.
	\end{equation}
	Therefore, $u$ is decreasing. We consider two cases. \\
	\noi \textbf{Case 1:} Let $-1<\al\leq 0$. In this case, $|x|^{n\al}e^{nu}$ is decreasing as $u$ is so. Thus using the hypothesis $(Q-M|x|^p)^+\in L^1(\R^n)$, we have 
	\begin{equation}
		I =\int_{B_1^c} (Q-M|x|^p)|x|^{n\al}e^{nu}\dx\leq e^{nu(1)}\int_{B_1^c} (Q-M|x|^p)^+\dx\leq C e^{nu(1)}. 
	\end{equation}
	Moreover, since $|x|^{n\al}e^{nu}$ is decreasing, \eqref{bol_1} implies that $$I\leq Ce^{nu(1)}\leq \frac{C}{|B_1|}\int_{B_1}|x|^{n\al}e^{nu(x)}\dx\leq C \La_1(1+\al).$$
	Thus, substituting the above estimate in \eqref{QM}, we obtain 
	\begin{equation}\label{QM1}
		\int_{\R^n} (Q-M|x|^p)|x|^{n\al}e^{nu}\dx\leq C.
	\end{equation}
	Therefore, using \eqref{bol_2} and \eqref{QM1}, we conclude
	\begin{align}
		\La_*=\int_{\R^n} Q|x|^{n\al}e^{nu}\dx &=\int_{\R^n} (Q-M|x|^p)|x|^{n\al}e^{nu}\dx+M\int_{\R^n}|x|^{n\al+p}e^{nu}\dx\\
		&\leq C+M\La_1\left(1+\frac{n\al+p}{n}\right)\\
		&= C\left(n,p,M,\al,\|Q\|_{L^\infty(B_1)},\|(Q-M|x|^p)^+\|_{L^1(B_1^c)}\right).
	\end{align} 
	
	\noi \textbf{Case 2:} Let $\al>0.$ Then using Lemma \ref{lem:enu_bound}, one gets
		\begin{equation}
			I=\int_{B_1^c} (Q-M|x|^p)|x|^{n\al}e^{nu}\dx\leq C\int_{B_1^c} (Q-M|x|^p)^+\dx\leq C.
		\end{equation}
		 Hence, proceeding in the same way as in \textbf{Case 1}, and using the above relation, we deduce 
		 \begin{equation}
		 	\La_*=\int_{\R^n} Q|x|^{n\al}e^{nu}\dx\leq C.
		 \end{equation}
		Combining both the cases, we conclude the proof.
\end{proof}

\begin{remark}\label{example:Q}  We now present two examples illustrating the necessity of the assumptions on 
$Q$ in Theorem  \ref{total-curvature}.   

	(i) Define $Q:\R^2\rightarrow\R$ by 
	\begin{equation}\label{eq:def-trial-state}
		Q(x)=\begin{cases}
			(|x|-1)(2-|x|), &\text{if}\;1<|x|< 2,\\
			0, & \text{otherwise}.
		\end{cases}
	\end{equation}
	For $Q$ as defined above and given $k\in \R$, let $u_k$ be a radial solution  to (existence follows from the ODE theory)
	\begin{equation}\label{eq:example}
		\left\{\begin{aligned}
			-\De u_k &= Q e^{2u_k}\quad \text{in}\; \R^2,\\
			u_k(0) &=k,\, u_k'(0)=0.
		\end{aligned}\right.
	\end{equation}
	Notice that $u_k$ is harmonic in $B_1$,  and in fact  $u_k\equiv k$ in $B_1.$ On the other hand, from \eqref{rad-relation} 
	\begin{equation}\label{example:u-mono}
		u_k(r)-u_k(1)=-C\int_{1}^{r}\frac{1}{s}\int_{B_s} Q e^{2u_k} \dx\d s,\quad r\geq1.
	\end{equation}
	We claim that
	\begin{equation}\label{examp:claim}
		\int_{\R^2} Q 
		 e^{2u_k(x)}\dx\rightarrow +\infty\;\quad\text{as}\;k\rightarrow+\infty.
	\end{equation}
	If not, then we can find $M>0$   such that 
	$$\int_{\R^2} Q 
	 e^{2u_k(x)}\dx\leq M\;\quad\text{for}\;k\geq 1.$$
	Thus, \eqref{example:u-mono} yields $|u_k(r)-u_k(1)|\leq C M$ for $1\leq r\leq 2.$ Therefore, one has $k-CM\leq u_k(r)$ for $1\leq r\leq 2.$ Using this lower bound on $u_k,$ one obtains 
	\begin{align}
	\int_{\frac{4}{3} <|x| <\frac{5}{3}} Q 
	e^{2u_k(x)}\dx & \geq C \int_{\frac{4}{3} <|x| <\frac{5}{3}} 
	e^{2u_k(x)}\dx\\
	& \geq C e^{2k-2CM}\rightarrow +\infty \;\text{as}\;k\rightarrow+\infty,
	\end{align} 
	a contradiction. Thus the claim follows.
	
	(ii) 
For $k\geq 1,$ let us consider continuous functions $Q_k:\R^2\rightarrow\R$ given by
	\begin{equation}\label{eq:def-trial-state2}
		Q_k(x)=\begin{cases}
			1, &\text{if}\;x\in B_1\cup B_2^c,\\
			2k(|x|-1)+1, &\text{if}\;1\leq |x|<\frac{3}{2},\\
			-2k(|x|-2)+1, &\text{if}\;\frac{3}{2}\leq |x|<2.
		\end{cases}
	\end{equation}
	For such $Q_k,$  let $u_k$ be a radial solution of \begin{equation}\label{eq:example2}
		\left\{\begin{aligned}
			-\De u_k &= Q_k e^{2u_k}\quad \text{in}\; \R^2,\\
			\int_{\R^2} Q_k 
			& e^{2u_k(x)}\dx < \infty,\\
			u_k(0) &=\log(2),\, u_k'(0)=0.
		\end{aligned}\right.
	\end{equation}
	We claim that
	\begin{equation}\label{examp:claim2}
		\int_{\R^2} Q_k 
		e^{2u_k(x)}\dx\rightarrow +\infty\;\quad\text{as}\;k\rightarrow+\infty.
	\end{equation}
	On the contrary, we can find $M>0$   such that 
	\begin{equation}\label{assum1}
		\int_{\R^2} Q_k 
		e^{2u_k(x)}\dx\leq M\;\quad\text{for}\;k\geq1.
	\end{equation}
	Observe that $u_k(x)=\log\left(\frac{2}{1+|x|}\right)$ in $B_1.$ Therefore, by  \eqref{example:u-mono}   and \eqref{assum1}, we obtain $u_k(r)\geq -C$ for $1\leq r\leq 2$ and for each $k\geq 1.$ Then,
	\begin{align}
		\int_{1 <|x| <2} Q_k 
		e^{2u_k(x)}\dx & \geq C \int_{1 <|x| <2} 
		Q_k\dx \rightarrow +\infty \;\text{as}\;k\rightarrow+\infty,
	\end{align} 
	which contradicts  \eqref{assum1}. Thus \eqref{examp:claim2} follows.

\end{remark}


\noi \textbf{Acknowledgments.} A. Hyder acknowledges the support from SERB SRG/2022/001291. M. Ghosh is supported by TIFR Centre for Applicable Mathematics (TIFR-CAM).

\noi \textbf{Conflict of interest statement.} The authors declare no potential conflicts of interest.

\noi\textbf{Data availability statement.} No data were used for the research described in this article. Therefore, data sharing is not applicable.

\bibliographystyle{abbrvurl}
\bibliography{Reference}
\end{document}